\documentclass[reqno]{amsart}

\usepackage{hyperref}
\usepackage{amsmath,amsthm,amssymb,amscd}
\usepackage{todonotes}

\usepackage{tikz}
\usetikzlibrary{angles,calc,arrows.meta}

\usepackage[
  style=numeric,
  sorting=nyt,
  giveninits=true,
  doi=false,
  url=false,
  isbn=false,
  maxnames=10,
]{biblatex}
\DeclareFieldFormat[article,incollection,inproceedings,misc]{title}{#1}
\DeclareFieldFormat[article,incollection,inproceedings,misc]{volume}{\mkbibbold{#1}}
\DeclareFieldFormat{pages}{#1}
\renewbibmacro{in:}{}
\AtEveryBibitem{\ifentrytype{book}{\clearfield{pages}}{}}
\AtEveryBibitem{\clearfield{eprintclass}}
\AtEveryBibitem{\clearlist{language}}

\newtheorem{thm}{Theorem}
\newtheorem{prop}[thm]{Proposition}
\newtheorem{cor}[thm]{Corollary}
\newtheorem{lem}[thm]{Lemma}

\theoremstyle{definition}
\newtheorem{dfn}[thm]{Definition}
\newtheorem{rem}[thm]{Remark}

\newtheorem{ques}[thm]{Question}

\theoremstyle{remark}
\newtheorem*{org}{Organization}
\newtheorem*{ack}{Acknowledgements}

\numberwithin{thm}{section}
\numberwithin{equation}{section}

\DeclareMathOperator{\In}{Int}
\DeclareMathOperator{\Conv}{Conv}
\DeclareMathOperator{\sgn}{sgn}

\title{Busemann G-spaces with convex balls}
\author[T. Fujioka]{Tadashi Fujioka}
\address[T. Fujioka]{Department of Applied Mathematics, Fukuoka University, Fukuoka 814-0180, Japan}
\email{tfujioka210@gmail.com}

\author[S. Gu]{Shijie Gu}
\address[S. Gu]{Department of Mathematics, Northeastern University, Shenyang, Liaoning, China, 110004}
\email{shijiegutop@gmail.com}

\date{\today}

\begin{document}

\begin{abstract}
We prove that any Busemann G-space such that every sufficiently small metric ball is convex is a topological manifold.
The key ingredient in the proof is Ivanov's Helly theorem.
The appendix contains a counterexample to a question of Berestovskii--Halverson--Repov\v s.
\end{abstract}

\maketitle

\section{Introduction}\label{sec:intro}

A \emph{G-space}, introduced by Busemann \cite{Bu55}, is a qualitative generalization of a Finsler manifold in terms of the geometry of geodesics.
It is a complete, locally compact geodesic space such that every small shortest path admits a local unique extension.
See Section \ref{sec:g} for the precise definition.
There are two long-standing problems concerning G-spaces, proposed by Busemann himself \cite[pp.403,2,49]{Bu55}:

\begin{enumerate}
\item Is any G-space finite-dimensional? (in the sense of topological dimension)
\item Is any finite-dimensional G-space a topological manifold?
\end{enumerate}

Both conjectures remain unsolved in their full generality.
See \cite{An18} for a survey and \cite{HR08} for the topological background.
The manifold conjecture (2) is known to hold in dimensions up to four (\cite{Bu55,Krak68,Th96}).
Moreover, it is relatively easy to verify both conjectures under \emph{quantitative} assumptions, such as synthetic curvature bounds, which further yield a Riemannian or Finsler structure.
See \cite{FGfin} and references therein.
Here, we address more \emph{qualitative} assumptions; in other words, we do not impose any kind of differentiable structure or curvature bounds.

In this qualitative direction, essentially the only general result known to us is Berestovskii's solution \cite{Be77} to the finite-dimensionality conjecture (1) under the assumption of the convexity of metric balls (see also a generalization by Berestovskii--Halverson--Repov\v s \cite{BHR11}).
The convexity of metric balls can be viewed as a qualitative generalization of non-positive curvature (cf.\ \cite{Bu48,An14,An17nor,FGtop}).
In this paper, we improve on Berestovskii's result to resolve the manifold conjecture (2) under the same (in fact, slightly weaker) convexity assumption.

\begin{thm}\label{thm:main}
Let $X$ be a G-space such that every sufficiently small metric ball is convex.
Then $X$ is a topological manifold.
\end{thm}

Here, the assumption means that for any $p\in X$, there exists $r>0$ such that every closed ball around $p$ of radius $\le r$ is convex.
In fact, it is enough to assume this condition for some nonempty open subset of $X$.
Furthermore, since we do not assume uniformity of the radius $r$, this is slightly weaker than Berestovskii's original assumption.
See Section \ref{sec:conv} for more details.

It is known that in general metric balls in a G-space need not be convex (\cite{BP80,BHR11}).
We also give such an example in Appendix \ref{sec:app}.
Thus Theorem \ref{thm:main} does not give a solution to the general conjecture.

The proof of Theorem \ref{thm:main} is outlined as follows.
Since any G-space is topologically homogeneous (\cite{Th96, BHR11}), to prove the manifold conjecture, it suffices to find one manifold point.
Our construction of a manifold point is indeed a refinement of the finite-dimensionality argument by Berestovskii \cite{Be77}.
We first choose sufficiently many points $\{p_i\}_{i=1}^N\subset X$ such that the distance map
\[(d(p_i,\cdot))_{i=1}^N:X\to\mathbb R^N\]
gives a local topological embedding, as Berestovskii did.
Then we use Ivanov's Helly theorem \cite{Iv14} (see Theorem \ref{thm:hel}) to find an $(n+1)$-point subset $I$ of $\{1,\dots,N\}$, where $n$ is the topological dimension of $X$, such that the corresponding distance map with one component removed
\[(d(p_i,\cdot))_{i\in I\setminus\{i_0\}}:X\to\mathbb R^n\]
gives the desired coordinates at some point.
Here, in addition to the Helly theorem, we make use of a Baire-type covering argument and the invariance of domain for homology manifolds (any finite-dimensional G-space is known to be an ANR integral homology manifold; \cite{Th96, BHR11}).

Geometrically speaking, what we are doing here is nothing other than finding $n+1$ points that span a nondegenerate $n$-simplex.
See Section \ref{sec:prf}, especially Remark \ref{rem:prf}.
Related arguments in the setting of synthetic upper curvature bounds can be found in \cite{Kl99, AKP24}.

Combining the above construction with the results of Berestovskii--Halverson--Repov\v s \cite{BHR11}, we obtain the following corollary.

\begin{cor}\label{cor:sph}
Let $X$ be a G-space as in Theorem \ref{thm:main}.
Then for any $p\in X$, every sufficiently small metric sphere around $p$ is homeomorphic to a sphere.
\end{cor}

The above corollary immediately implies the following global result.
Recall that a G-space is \emph{straight} if every shortest path is extendable to a line.

\begin{cor}\label{cor:glo}
Let $X$ be a straight G-space such that every metric ball is convex.
Then $X$ is homeomorphic to Euclidean space.
\end{cor}

As the proof shows, the assumption that all metric balls are convex is superfluous here (local convexity as in Theorem \ref{thm:main} is enough).
The reason we formulate it this way is that it can be viewed as a generalization of the Cartan--Hadamard theorem, and also because many equivalent conditions are known under the assumption of straightness; see \cite[p.121, Theorem 20.9]{Bu55}.

\begin{rem}
Theorem \ref{thm:main}, as well as Corollaries \ref{cor:sph} and \ref{cor:glo}, generalize the corresponding results for G-spaces with nonpositive curvature in the sense of Busemann.
These results were originally due to Andreev \cite{An14,An17nor}, and an alternative proof was recently obtained by the authors \cite[Theorems 1.6 and 7.13]{FGtop}.
\end{rem}

\begin{rem}
In \cite[Theorem 1.5]{FGfin}, the authors also proved that any G-space such that the squared distance function to a point is locally semiconvex is a topological manifold.
This is another generalization of the case of nonpositive curvature mentioned above.
However, in general, neither of the assumptions in Theorem 1.1 and \cite[Theorem 1.5]{FGfin} implies the other (at least not immediately). Nevertheless, common techniques appear in their proofs; a Berestovskii-type argument concerning finite-dimensionality, as well as a Baire category argument, had already appeared in the proof of \cite[Theorem 1.5]{FGfin}.
It is worth emphasizing that these observations indirectly led to the present results.
\end{rem}

\begin{org}
In Section \ref{sec:pre}, we briefly discuss the definition and properties of a G-space, convexity of its metric balls, and Ivanov's Helly theorem \cite{Iv14}.
In Section \ref{sec:prf}, we prove Theorem \ref{thm:main} and Corollaries \ref{cor:sph} and \ref{cor:glo}.
In Appendix \ref{sec:app}, we give a counterexample to Question 9.1 of Berestovskii--Halverson--Repov\v s \cite{BHR11}.
\end{org}

\begin{ack}
The authors would like to thank Valerii Berestovskii for his interest in this work and for valuable comments.
The first author was supported by JSPS KAKENHI Grant Number 25K23336.
The second author was supported by NSFC grant 12201102.
The authors used artificial intelligence during the exploratory stage of this research. A detailed statement on the use of AI is provided at the end of the paper.
\end{ack}

\section{Preliminaries}\label{sec:pre}

The distance between $x$ and $y$ is denoted by $d(x,y)$.
A shortest path (see below) between $x$ and $y$ is denoted by $xy$.
The open and closed $r$-ball around $p$ are denoted by $B(p,r)$ and $\bar B(p,r)$, respectively.
The interior and boundary of a subset $A$ are denoted by $\In A$ and $\partial A$, respectively.

\subsection{G-spaces}\label{sec:g}

We first recall the definition and properties of a G-space.
Although our definition differs from the original axiomatic definition in \cite{Bu55}, they are known to be equivalent.
See also \cite{Th96,BHR11} for basic properties.

A \emph{shortest path} is an isometric embedding of an interval into a metric space.
A metric space is a \emph{geodesic space} if any two points are joined by a shortest path.
A geodesic space is \emph{uniquely geodesic} if any two points are joined by a unique shortest path.
Similarly, a geodesic space is \emph{locally uniquely geodesic} if every point has a neighborhood $U$ such that any two points in $U$ are joined by a unique shortest path (possibly outside $U$).

\begin{dfn}\label{dfn:g}
A \emph{G-space} is a complete, locally compact geodesic space with a local unique extension property.
\end{dfn}

The last condition means that every point has a neighborhood $U$ such that any shortest path in $U$ admits a unique maximal extension reaching $\partial U$.
In particular, this implies the local geodesic uniqueness mentioned above.
Furthermore, by the Arzel\`a--Ascoli theorem, such locally unique shortest paths vary continuously depending on their endpoints.
Note also that any G-space is \emph{proper}, i.e., every bounded closed subset is compact.

The following topological properties will be used later.
Every G-space is \emph{topologically homogeneous}, i.e.,  any point can be moved to any other point by a self-homeomorphism of the space (\cite[Theorem 2.5]{Th96}, \cite[Corollary 3.12]{BHR11}).
Moreover, every finite-dimensional G-space is an ANR integral homology manifold (\cite[Corollary 2.8, Theorem 2.12]{Th96}, \cite[Proposition 3.13, Theorem 3.16]{BHR11}).
Since the latter is used solely to apply the invariance of domain, its definition is omitted here.

\subsection{Convex balls}\label{sec:conv}

Next we discuss convexity assumptions for metric balls in a G-space.
See \cite[Section 20]{Bu55} for basic properties and \cite{BP79,BP80,BP83} for related works.
See also \cite{Fo04} for several convexity notions in a more general setting.

Let $X$ be a geodesic space.
A subset $A$ of $X$ is \emph{convex} if any shortest path $\gamma$ connecting two points of $A$ is contained in $A$.
Furthermore, $A$ is \emph{strictly convex} if the interior of $\gamma$ is contained in the interior of $A$.
In the setting of G-spaces, we only consider these conditions in a small neighborhood in which shortest paths are unique.

Recall that in Theorem \ref{thm:main}, we assumed the convexity of metric balls non-strictly and pointwise: the assumption there means that for any $p\in X$, there exists $r>0$ such that every closed ball around $p$ of radius $\le r$ is convex.
The following simple lemma improves it to a strong and uniform form.
This plays a fundamental role throughout the proof of Theorem \ref{thm:main}.

\begin{lem}\label{lem:conv}
Let $X$ be a G-space with convex balls as in Theorem \ref{thm:main}.
Then there exist a nonempty open subset $O\subset X$ and $R>0$ such that for any $p\in O$, any closed ball around $p$ of radius $\le R$ is strictly convex.
\end{lem}

\begin{proof}
We prove the claim with strict convexity replaced by non-strict convexity.
Then the strict version follows from \cite[p.119, (20.5)]{Bu55} (by possibly replacing $O$ and $R$ with smaller ones).

Choose an open ball $B\subset X$ and $a>0$ such that any shortest path connecting two points in the $a$-neighborhood of $B$ is unique (in $X$).
For any integer $j\ge 1/a$, put
\[F_j:=\{p\in B\mid\text{$\bar B(p,r)$ is convex for every $0<r\le1/j$}\}.\]
In particular, any shortest path connecting two points in $\bar B(p,r)$ is unique.
By the pointwise convexity assumption, the sets $F_j$ cover $B$.

We show that $F_j$ is a closed subset of $B$.
Suppose $p_k\in F_j$ converges to $p\in B$ and let $x,y\in\bar B(p,r)$ with $0<r\le 1/j$. Moving $x,y$ slightly toward $p_k$ along shortest paths, we can find $x_k,y_k\in\bar B(p_k,r)$ that converge to $x,y$, respectively. By the definition of $F_j$, the unique shortest path $x_ky_k$ is contained in $\bar B(p_k,r)$. Since $x_ky_k$ converges to the unique shortest path $xy$, this shows that $\bar B(p,r)$ is convex, as desired.

Since $B$ is a locally compact metric space, by the Baire category theorem, some $F_j$ must have nonempty interior in $B$ (and hence in $X$).
Setting $O:=\In F_j$ and $R:=1/j$ proves the claim.
\end{proof}

\begin{rem}
For the proof of Theorem \ref{thm:main}, we only need the conclusion of Lemma \ref{lem:conv}.
For the proof of Lemma \ref{lem:conv}, we only need the pointwise convexity assumption on some nonempty open subset.
Therefore, the assumption of Theorem \ref{thm:main} can be relaxed to the following form: there exists a nonempty open subset $O'\subset X$ such that for any $p\in O'$, there exists $r>0$ such that every closed ball around $p$ of radius $\le r$ is convex.
\end{rem}

\begin{rem}
Berestovskii's original proof of finite-dimensionality \cite{Be77} assumes the conclusion of Lemma \ref{lem:conv}.
Thus Lemma \ref{lem:conv} together with the previous remark slightly improves it.
It is worth emphasizing that a similar Baire-category argument will be used again in the proof of Theorem \ref{thm:main} (see Section \ref{sec:mani}).
\end{rem}

\subsection{Ivanov's Helly theorem}

Finally, we recall the Helly-type theorem for geodesic spaces proved by Ivanov \cite{Iv14}.
This is the key ingredient in the proof of Theorem \ref{thm:main}.

\begin{thm}\label{thm:hel}
Let $X$ be a proper, uniquely geodesic space of compact topological dimension (at most) $n<\infty$.
Let $\{A_i\}$ be a finite collection of convex subsets of $X$ such that every subcollection of cardinality $\le n+1$ has nonempty intersection.
Then $\bigcap A_i\neq\emptyset$.
\end{thm}

Here the \emph{compact topological dimension} of $X$ is the supremum of the topological dimension of all compact subsets of $X$.
We emphasize that the precise bound $n+1$ is essential for our application.

\begin{rem}
The original statement in \cite[Theorem 1.1]{Iv14} does not include the properness assumption.
The reason we include it here is that Ivanov assumed the continuous dependence of shortest paths on their endpoints in his definition of a uniquely geodesic space, which can be guaranteed by properness as he mentioned in \cite[p.110]{Iv14}.
\end{rem}

\begin{rem}
Theorem \ref{thm:hel} has a more topological formulation; see \cite[Proposition 2.2]{Iv14}.
The uniqueness of geodesics and the convexity of subsets are geometric assumptions that guarantee the topological assumptions there.
\end{rem}

\section{Proof}\label{sec:prf}

In this section, we prove Theorem \ref{thm:main} and Corollary \ref{cor:glo}.
In what follows, unless otherwise stated, $X$ denotes a G-space with convex balls as in Theorem \ref{thm:main}.

\subsection{Finite-dimensionality}

We first recall the proof of finite-dimensionality by Berestovskii \cite{Be77} (cf.\ \cite{BHR11}).
In fact, the construction used there is necessary to prove Theorem \ref{thm:main}.

Let $X$ be as in Theorem \ref{thm:main}.
By the local unique extension property (see Section \ref{sec:g}) and Lemma \ref{lem:conv}, we may assume that there exist $o\in X$ and $R>0$ satisfying the following:
\begin{itemize}
\item Every shortest path of length less than $R$
whose image lies in $\bar B(o,R)$ is uniquely extendable to length $R$.
\item Every closed ball contained in $\bar B(o,R)$ is strictly convex.
\end{itemize}
In what follows, we will always work within the $R$-neighborhood of $o$ and frequently use the above two properties.

Fix $0<r<R/100$ and consider a closed convex ball
\[\bar B:=\bar B(o,r).\]
Take a finite $r$-net $\{p_i\}_{i=1}^N$ in $\bar B(o,10r)$, i.e., a set of points whose $r$-neighborhoods cover $\bar B(o,10r)$.

\begin{lem}\label{lem:fin}
In the above situation, let $x,y\in\bar B$ be any pair of distinct points.
Then there exists $1\le i\le N$ such that
\begin{equation}\label{eq:fin}
d(p_i,y)>d(p_i,x).
\end{equation}
\end{lem}

\begin{proof}
Extend the shortest path $yx$ beyond $x$ to length $9r$ and let $z\in\bar B(o,10r)$ be the new endpoint.
Choose $p_i$ that is $r$-close to $z$.
By the triangle inequality, we have
\[d(p_i,z)<d(p_i,y).\]
In particular, this implies that $z$ is contained in $\bar B(p_i,d(p_i,y))$.
Since this ball is strictly convex and $x$ lies in the interior of the shortest path $yz$, we obtain the desired inequality \eqref{eq:fin}.
\end{proof}

\begin{rem}\label{rem:fin}
The above lemma implies that the distance map
\[F:=(d(p_i,\cdot))_{i=1}^N:\bar B\to\mathbb R^N\]
is a topological embedding.
Together with the topological homogeneity, this shows the finite-dimensionality of $X$.
Indeed, this was the original proof by Berestovskii \cite{Be77}.
For this argument, we only need $d(p_i,y)\neq d(p_i,x)$ rather than the oriented inequality \eqref{eq:fin}.
On the other hand, this oriented inequality plays an important role in the proof of manifoldness below.
\end{rem}

\subsection{Manifoldness}\label{sec:mani}

We proceed with the proof of Theorem \ref{thm:main}.
We will prove several key lemmas.

For any nonempty $I\subset\{1,\dots,N\}$ and $x\in\bar B$, put
\[C_I(x):=\bar B\cap\bigcap_{i\in I}\bar B(p_i,d(p_i,x)).\]
In other words, $C_I(x)$ is the set of points in $\bar B$ whose distance from $p_i$ ($i\in I$) is less than or equal to that of $x$.
Clearly $x\in C_I(x)$.
Lemma \ref{lem:fin} shows that $C_I(x)=\{x\}$ for the full index set $I=\{1,\dots, N\}$.

The following lemma is key to the proof of Theorem \ref{thm:main}; this is the only place where we use Ivanov's Helly theorem \ref{thm:hel}.
Let $n$ be the topological dimension of $X$, which is finite by Remark \ref{rem:fin}.

\begin{lem}\label{lem:hel}
For any $x\in\bar B$, there exists a nonempty subset $I\subset\{1,\dots,N\}$ with $|I|\le n+1$ such that
\[C_I(x)=\{x\}.\]
\end{lem}

\begin{proof}
For each $1\le i\le N$, put
\[A_i:=\bar B\cap(\bar B(p_i,d(p_i,x))\setminus\{x\}).\]
Since $\bar B(p_i,d(p_i,x))$ is strictly convex, removing a boundary point preserves the convexity; in particular, $\bar B(p_i,d(p_i,x))\setminus\{x\}$ is convex.
Since $\bar B$ is also convex, it follows that $A_i$ is a convex subset.
Furthermore, by Lemma \ref{lem:fin}, the intersection of all $A_i$ is empty.
Note also $\dim\bar B\le n$ (in fact the equality holds by topological homogeneity).
Thus, applying the contrapositive of Ivanov's Helly theorem \ref{thm:hel} in the compact uniquely geodesic space $\bar B$ implies that there exists $I\subset\{1,\dots,N\}$ with $|I|\le n+1$ such that
\[\bigcap_{i\in I}A_i=\emptyset.\]
Since $\bigcap_{i\in I}A_i=C_I(x)\setminus\{x\}$, this shows the claim.
\end{proof}

For any nonempty $I\subset\{1,\dots,N\}$, put
\[H_I:=\{x\in\bar B\mid C_I(x)=\{x\}\}.\]
In other words, $H_I$ is the set of points in $\bar B$ for which the conclusion of Lemma \ref{lem:hel} holds for $I$.
Lemma \ref{lem:hel} shows that the sets $H_I$ with $|I|\le n+1$ cover $\bar B$; this fact will be used later.

\begin{lem}\label{lem:clo}
The subset $H_I$ is closed in $\bar B$.
\end{lem}

\begin{proof}
Suppose $x_k\in H_I$ converges to $x\in\bar B$.
We show $x\in H_I$.
Suppose $x\notin H_I$; then there exists $y\in C_I(x)\setminus\{x\}$.
Let $z$ be any interior point of the shortest path $xy$.
By the convexity of $\bar B$, we have $z\in\bar B$.
Furthermore, by the assumption $y\in C_I(x)$ and the strict convexity of $\bar B(p_i,d(p_i,x))$, we have
\[d(p_i,z)<d(p_i,x)\]
for all $i\in I$.
Since $x_k$ converges to $x$, this implies that $z\in C_I(x_k)\setminus\{x_k\}$ for any sufficiently large $k$.
This contradicts $x_k\in H_I$.
\end{proof}

Finally, we introduce the distance coordinates on $H_I$.
For any nonempty $I\subset\{1,\dots,N\}$ and fixed $i_0\in I$, define
\[F_{I,i_0}:=(d(p_i,\cdot))_{i\in I\setminus\{i_0\}}:H_I\to\mathbb R^m,\]
where $|I|=m+1$.

\begin{lem}\label{lem:emb}
The map $F_{I,i_0}$ is a topological embedding.
\end{lem}

\begin{proof}
Since $F_{I,i_0}$ is continuous and the domain $H_I$ is compact by Lemma \ref{lem:clo}, it suffices to show that $F_{I,i_0}$ is injective.
Let $x\neq y\in H_I$.
By the definition of $H_I$, there exist distinct $i,j\in I$ such that
\[d(p_i,y)>d(p_i,x),\quad d(p_j,y)<d(p_j,x).\]
Since $i$ and $j$ are different, one of them must be different from $i_0$.
Thus we have $F_{I,i_0}(x)\neq F_{I,i_0}(y)$, as desired.
\end{proof}

Now we are in a position to prove Theorem \ref{thm:main}.

\begin{proof}[Proof of Theorem \ref{thm:main}]
As already mentioned, by Lemma \ref{lem:hel}, we have
\[\bar B=\bigcup_{\substack{\emptyset\neq I\subset\{1,\dots,N\}\\|I|\le n+1}}H_I,\]
where $n$ is the topological dimension of $X$.
By Lemma \ref{lem:clo}, $H_I$ is closed in $\bar B$.
By the Baire category theorem (or by arguing directly since the union is finite), one of the above $H_I$ has nonempty interior in $\bar B$, and hence in $X$.
Fix such an index subset $I$ and let $U$ be the interior of $H_I$ in $X$.
By Lemma \ref{lem:emb}, the map
\[F_{I,i_0}|_U:U\to\mathbb R^m\]
is a topological embedding, where $i_0\in I$ and $|I|=m+1$.

Since $|I|\le n+1$ as stated above, we have $m\le n$.
On the other hand, since $X$ is topologically homogeneous, $\dim U=n$.
Since $F_{I,i_0}|_U$ is a topological embedding, we have $n\le m$.
Therefore, $m=n$.

Recall that $X$ is a finite-dimensional, ANR integral homology manifold (\cite{Th96, BHR11}).
The invariance of domain theorem for such homology manifolds shows that the image of $F_{I,i_0}|_U$ is open in $\mathbb R^n$ (see \cite[p.383, Corollary V.16.19; p.375, Theorem V.16.8]{Bre97} for more details; see also \cite[p.122]{Bre97} for the equivalence of dimensions).
This gives a manifold point of $X$.
Since $X$ is topologically homogeneous, $X$ is a topological manifold.
\end{proof}

We give some geometric interpretation of the above proof.

\begin{rem}\label{rem:prf}
We consider the Euclidean case.
Let $p_i\in\mathbb R^n$ be arbitrary points, where $1\le i\le N$.
For any nonempty $I\subset\{1,\dots,N\}$ and $x\in\mathbb R^n$, put
\[D_I(x):=\bigcap_{i\in I}\bar B(p_i,|p_i-x|),\]
where $|\cdot|$ denotes the Euclidean norm.
Here, for simplicity, we do not introduce a fixed ball $\bar B$ as before.
Then it is easy to see that
\[D_I(x)=\{x\}\iff x\in\Conv\{p_i\mid i\in I\},\]
where $\Conv(\cdot)$ denotes the convex hull.
Indeed, if $x\in\Conv\{p_i\mid i\in I\}$, then moving $x$ in any direction increases at least one of the distances from $p_i$.
Conversely, if $x\notin\Conv\{p_i\mid i\in I\}$, then one can find a direction at $x$ that decreases all of the distances from $p_i$.

Therefore, if $D_{\{1,\dots,N\}}(x)=\{x\}$ for the full index set, $x$ is contained in the convex hull of all $p_i$.
Then by the Carath\'eodory theorem, one can find $I\subset\{1,\dots,N\}$ with $|I|\le n+1$ such that $x$ is contained in the convex hull of $p_i$ with $i\in I$, and hence $D_I(x)=\{x\}$.
Lemma \ref{lem:hel} based on Ivanov's Helly theorem is a generalization of this argument of finding a simplex with at most $n+1$ vertices.
The Baire-category argument in the proof of Theorem \ref{thm:main} then allows us to find a nondegenerate one, i.e., $|I|=n+1$.
\end{rem}

Finally, we prove Corollaries \ref{cor:sph} and \ref{cor:glo}.

\begin{proof}[Proof of Corollary \ref{cor:sph}]
Let $X$ be a G-space as in Theorem \ref{thm:main} and let $n$ be the topological dimension of $X$.
By the proof of Theorem \ref{thm:main}, there exist $x\in X$ and $p_1,\dots,p_n\in X$ near $x$ such that the distance map
\[F:=(d(p_1,\cdot),\dots,d(p_n,\cdot)):X\to\mathbb R^n\]
is an open topological embedding around $x$.
Consider a map with one component removed,
\[G:=(d(p_2,\cdot),\dots,d(p_n,\cdot)):X\to\mathbb R^{n-1},\]
and restrict it to a metric sphere $S$ through $x$ centered at $p_1$.
Then it shows that $x$ is a manifold point in $S$.

By the local convexity assumption, $X$ is \emph{locally G-homogeneous} in the sense of Berestovskii--Halverson--Repov\v s \cite{BHR11}; see Definition \ref{dfn:app}.
One of their results \cite[Theorem 1.1]{BHR11} says that $S$ is topologically homogeneous (we may assume that the radius of $S$ is sufficiently small so that the above result can be applied; the same convention applies below).
Since $S$ contains a manifold point, it is a topological manifold.
Moreover, as shown by the second author \cite{Gu19}, $S$ is homotopy equivalent to a sphere.
By the solution of the (generalized) Poincar\'e conjecture, $S$ is homeomorphic to a sphere.
Another result of Berestovskii--Halverson--Repov\v s \cite[Theorem 5.6]{BHR11} shows that all sufficiently small metric spheres in $X$ are homeomorphic.
This completes the proof.
\end{proof}

\begin{proof}[Proof of Corollary \ref{cor:glo}]
Let $X$ be a straight G-space with convex balls as in Corollary \ref{cor:glo}.
Fix $o\in X$ and suppose $r>0$ is sufficiently small.
By Corollary \ref{cor:sph} and the local cone structure of a G-space (\cite[Proposition 2.2]{Th96}, \cite[Proposition 3.3]{BHR11}), $\bar B(o,r)$ is homeomorphic to a closed Euclidean ball.
By the straightness assumption, $B(o,r)$ is homeomorphic to $X$: indeed, let $\phi:[0,r)\to[0,\infty)$ be a homeomorphism and consider a map
\[h:B(o,r)\to X\]
that sends any $x\in B(o,r)$ to the point at distance $\phi(d(o,x))$ from $o$ on the unique extension of the shortest path $ox$ beyond $x$.
This completes the proof.
\end{proof}

\begin{rem}
In fact, for the proof of Corollary \ref{cor:glo}, Theorem \ref{thm:main} is sufficient and Corollary \ref{cor:sph} is unnecessary.
Once a Euclidean neighborhood $V$ of $o\in X$ is found, one can enlarge it via a map similar to $h$ to construct an exhaustive sequence of open subsets $V_1\subset V_2\subset\cdots$ of $X$ that are homeomorphic to Euclidean space.
Then Brown's monotone-union theorem \cite{Br61} shows that $X$ itself is Euclidean.
\end{rem}

We conclude this section by proposing a potential generalization.

\begin{rem}
It is worth noting that the uniqueness of the extension of shortest paths was not used in most of the proof of Theorem \ref{thm:main}, specifically from Lemma \ref{lem:fin} to Lemma \ref{lem:emb}; we only used the strict convexity of metric balls and the existence of the extension.
The uniqueness of the extension only enters when we use the topological homogeneity and the homology manifold structure of a G-space in the final step of the proof of Theorem \ref{thm:main} (and also in Lemma \ref{lem:conv} at the beginning of the proof to upgrade the convexity to the strict one).
This suggests the possibility of developing a theory of singular spaces with strictly convex balls, which includes the present setting of G-spaces with convex balls as a special case.
See \cite[Section 1.2]{FGfin} and compare with the theory of GCBA/GNPC spaces in \cite{LN19,LN22,FGtop,FGgeo}.
However, we will not pursue this direction further in this paper.
\end{rem}

\appendix

\section{}\label{sec:app}

In this appendix, we answer Question 9.1 of Berestovskii--Halverson--Repov\v s \cite{BHR11}.
They introduced weaker variants of the convexity of metric balls, called \emph{(uniform) local G-homogeneity}, and asked whether such conditions hold for every G-space.
We give a counterexample to this question.

We first recall their terminology.
Let $X$ be a geodesic space.
A subset $A$ of $X$ is \emph{starlike} with respect to $p\in\In A$ if $A$ is a geodesic cone over $\partial A$ with vertex $p$; that is, every $x\in A\setminus\{p\}$ lies on a unique shortest path contained in $A$ from $p$ to some point of $\partial A$.
Similarly, $A$ is \emph{stably starlike} at $p\in\In A$ if there exists $\delta>0$ such that $A$ is starlike with respect to any $x\in B(p,\delta)$.
As usual, in the setting of G-spaces, we consider these conditions in a small neighborhood in which shortest paths are uniquely extendable.

\begin{dfn}\label{dfn:app}
Let $X$ be a G-space.
We say that $X$ is \emph{locally G-homogeneous} if for every $p\in X$, there exists sufficiently small $\epsilon>0$ (smaller than the radius of local unique extension) such that $\bar B(p,\epsilon)$ is stably starlike at $p$.

We say that $X$ is \emph{uniformly locally G-homogeneous on an orbal set} if there exist a subset $C\subset X$ and sufficiently small $0<\delta\le \epsilon_1<\epsilon_2$ such that for every $p\in C$ and $\epsilon\in(\epsilon_1,\epsilon_2)$, $\bar B(p,\epsilon)$ is starlike with respect to any $q\in B(p,\delta)$, and $C$ contains a ball of radius greater than $\epsilon_2$.
\end{dfn}

Note that the second condition is not necessarily stronger than the first: while the second condition possesses uniformity, it concerns only the subset $C$.
However, both conditions are weaker than the convexity of metric balls (cf.\ Lemma \ref{lem:conv}).
Berestovskii--Halverson--Repov\v s \cite{BHR11} proved that for any local G-homogeneous G-space, any sufficiently small metric sphere is topologically homogeneous.
They also proved that any G-space that is uniformly locally G-homogeneous on an orbal set is finite-dimensional.
The latter generalizes the earlier result of Berestovskii \cite{Be77} (but the proof is essentially the same).

In \cite[Question 9.1]{BHR11}, Berestovskii--Halverson--Repov\v s asked whether every G-space is locally G-homogeneous or uniformly locally G-homogeneous on an orbal set.
We give a negative answer to this question by proving the following.

\begin{prop}\label{prop:app}
For any integer $n\ge 2$, there exist a straight G-space metric on $\mathbb R^n$ and a dense subset $S\subset\mathbb R^n$ such that for any $p\in S$ and any $\epsilon>0$, $\bar B(p,\epsilon)$ is not stably starlike at $p$.
\end{prop}

To prove it, we first consider a simple case where $S$ is a hyperplane
\[H:=\{x_1=0\},\]
where $x_1$ denotes the first coordinate of $\mathbb R^n$.
The modification to a dense subset will be discussed later.

Fix $0<\alpha<1$ and define a function
\begin{equation}\label{eq:phi}
\phi(s):=\sgn(s)|s|^\alpha,
\end{equation}
where $\sgn$ is the sign function.
Note that $\phi$ is continuous and strictly increasing.
We introduce a new metric $d_\phi$ on $\mathbb R^n$ by
\begin{equation}\label{eq:dphi}
d_\phi(x,y):=|x-y|+|\phi(x_1)-\phi(y_1)|,
\end{equation}
where $|\cdot|$ denotes the Euclidean norm and $x_1,y_1$ denote the first coordinates of $x,y\in\mathbb R^n$, respectively.
Clearly $d_\phi$ is a metric.

\begin{lem}\label{lem:app1}
The space $(\mathbb R^n,d_\phi)$ is a straight G-space such that its metric topology is the Euclidean topology and its geodesic segments are the affine segments (possibly with nonlinear parametrization).
\end{lem}

\begin{proof}
We check Definition \ref{dfn:g} along with the additional requirements.
Since $d_\phi$ is not less than the Euclidean distance, its metric topology is finer than the Euclidean topology.
Conversely, since $\phi$ is continuous, the Euclidean topology is finer than the metric topology.
Therefore, the two topologies are the same, and in particular $(\mathbb R^n,d_\phi)$ is locally compact.
Similarly, since any $d_\phi$-Cauchy sequence is a Euclidean Cauchy sequence, we see that $d_\phi$ is a complete metric.

Furthermore, since $\phi$ is strictly increasing, we have the following: for any distinct $x,y,z\in\mathbb R^n$, the triangle inequality
\[d_\phi(x,z)\le d_\phi(x,y)+d_\phi(y,z)\]
becomes an equality if and only if $y$ lies on the affine segment between $x$ and $z$.
This shows that the geodesic segments of $(\mathbb R^n,d_\phi)$ are the affine segments (possibly with nonlinear parametrization).
In particular, $(\mathbb R^n,d_\phi)$ satisfies the global unique extension property, and hence it is a straight G-space.
\end{proof}

We denote by $\bar B_\phi(\cdot,\cdot)$ a closed ball with respect to $d_\phi$.
Recall $H:=\{x_1=0\}$.
The key observation is the following.

\begin{lem}\label{lem:app2}
Let $p\in H$.
For any $\epsilon>0$, $\bar B_\phi(p,\epsilon)$ is not stably starlike at $p$.
\end{lem}

\begin{proof}
Let $e_1,e_2$ be the first two standard coordinate vectors in $\mathbb R^n$.
Fix arbitrary small $0<\delta\ll\epsilon$.
Set
\[q:=p+\delta e_1,\quad x:=p+\epsilon e_2.\]
Clearly $q,x\in\bar B_\phi(p,\epsilon)$ (in fact, $q$ is very close to $p$ whereas $x$ lies on the boundary).
For sufficiently small $t>0$, consider
\[x_t:=tq+(1-t)x=p+t\delta e_1+(1-t)\epsilon e_2,\]
i.e., a point on the shortest path $qx$ that is sufficiently close to $x$.
Then, by \eqref{eq:phi} and \eqref{eq:dphi}, we have
\begin{equation}\label{eq:dphit}
\begin{aligned}
d_\phi(p,x_t)&=\sqrt{(t\delta)^2+((1-t)\epsilon)^2}+\phi(t\delta)\\
&\ge (1-t)\epsilon+(t\delta )^\alpha.
\end{aligned}
\end{equation}
Since $0<\alpha<1$, if $t$ is small enough compared to $\epsilon$ and $\delta$, we have $d_\phi(p,x_t)>\epsilon$.
Since $\delta$ is arbitrary, this shows that $\bar B_\phi(p,\epsilon)$ is not stably starlike.
\end{proof}

\begin{rem}
The idea behind the above argument is as follows.
When moving from $x$ to $q$, the Euclidean distance to $p$ decreases at most linearly in $t$, whereas $\phi$ increases on the order of $t^\alpha$.
Consequently, the total distance increases.
\end{rem}

Now we prove Proposition \ref{prop:app}.

\begin{proof}[Proof of Proposition \ref{prop:app}]
We modify the above construction to obtain a dense set of non-stably starlike points.
Fix $0<\alpha<1$ and take an arbitrary countable dense subset $\{a_j\}_{j=1}^\infty$ of $\mathbb R$.
Define a function
\begin{equation}\label{eq:phi'}
\phi(s):=\sum_{j=1}^\infty2^{-j}\tanh(\sgn(s-a_j)|s-a_j|^\alpha)
\end{equation}
This function possesses similar properties to the previous one at each $a_j$.

\begin{lem}\label{lem:phi'}
The function $\phi$ is continuous and strictly increasing.
Furthermore, for every $j$ there exists $c_j>0$ such that
\begin{equation}\label{eq:phih}
\phi(a_j+h)-\phi(a_j)\ge c_jh^\alpha
\end{equation}
for any sufficiently small $h>0$.
\end{lem}

\begin{proof}
Since the series in \eqref{eq:phi'} converges uniformly, $\phi$ is continuous.
Since every summand in \eqref{eq:phi'} is strictly increasing, $\phi$ is strictly increasing.
For small $h>0$, by looking at the $j$-th summand, we have
\[\phi(a_j+h)-\phi(a_j)\ge2^{-j}\tanh(h^\alpha),\]
which shows the desired inequality.
\end{proof}

Now the preceding argument applies with the following modifications.
We introduce a metric on $\mathbb R^n$ by the same formula \eqref{eq:dphi} with this new $\phi$.
Then Lemma \ref{lem:app1} holds exactly as before, since the proof relies only on the continuity and strict monotonicity of $\phi$ proved in Lemma \ref{lem:phi'}.
Lemma \ref{lem:app2} also holds by replacing $H$ with $H_j:=\{x_1=a_j\}$ and using \eqref{eq:phih} in place of \eqref{eq:dphit}.
Since the union of $H_j$ is dense in $\mathbb R^n$, this completes the proof.
The details are left to the reader.
\end{proof}

The following natural question remains, but we will not pursue it here.

\begin{ques}
Is any G-space that is uniformly locally G-homogeneous on an orbal set a topological manifold?
More specifically, can the proof of Theorem \ref{thm:main} be adapted to this setting?
\end{ques}

\section*{Statement on the use of AI}

Artificial intelligence was used during the exploratory stage of this research. In particular, the authors used ChatGPT, including a GPT-5.6 Sol Ultra Codex configuration through VS Code, to investigate possible approaches to the Busemann conjecture and related research projects and questions. The AI was provided with relevant references, previous manuscripts, and mathematical context arising from the authors' earlier work and discussions, including Berestovskii's finite-dimensionality argument and approaches involving convexity, Baire-category arguments, and topological methods.

A Helly-type theorem had already appeared in an unpublished manuscript of one of the authors and was among the mathematical ideas available in this context, although it had not been connected there with Berestovskii's distance-coordinate construction or used to prove the convex-ball version of the Busemann conjecture. During a broader AI-assisted exploration of several possible approaches for about 56 hours, ChatGPT identified a way to combine this Helly-theoretic direction with Berestovskii's construction and suggested using Ivanov's Helly theorem to reduce the number of distance coordinates. This synthesis led to the main argument of the present paper.

The resulting proof was subsequently formulated, simplified, and checked by the authors. In particular, the authors verified the oriented distance inequality underlying the distance-coordinate construction, the application of Ivanov's Helly theorem, the Baire-category argument, the distance-coordinate embedding, and the use of invariance of domain for homology manifolds. AI tools were also used during the exploratory process to examine possible examples and counterexamples, identify connections with the existing literature, organize references, and check intermediate arguments.

The role of AI in this work should therefore be understood as part of the mathematical exploration and discovery process rather than as an independent source of the final proof. The mathematical background, relevant references and prior ideas, assessment of the AI-generated suggestions, and the final formulation and verification of the results were provided by the authors. The authors have independently reviewed the arguments in the manuscript and take full responsibility for the correctness and presentation of its contents.

\printbibliography

\end{document}